\documentclass[12pt]{amsart}
\usepackage{amssymb}

\newtheorem{thm}     {Theorem}[section]
\newtheorem{prop}    [thm]{Proposition}
\newtheorem{definition}  [thm]{Definition}

\newtheorem{lemma}   [thm]{Lemma}
\newtheorem{remark}   [thm]{Remark}
\newtheorem{example}   [thm]{Example}

\newcommand{\g}{\mathfrak g}

\newcommand{\C}{\mathbb C}

\newcommand{\R}{\mathbb R}
\newcommand{\Z}{\mathbb Z}

\newcommand{\Span}{{\rm Span}}
\def\d{\partial}
\def\l{\ell}
\def\Re{{\rm Re\,}}
\def\Im{{\rm Im\,}}
\def\bar{\overline}
\def\rank{{\rm rank}\,}
\def\<{\langle}
\def\>{\rangle}
\def\({\left(}

\def\){\right)}
\def\br{\hfill\break}
\def\ad{{\rm ad\,}}

\begin{document}


\title[Infinitesimal automorphisms
and second jet determination for CR maps]
{Infinitesimal automorphisms of quadrics and \\
second jet determination for CR mappings}

\author{Alexander Tumanov}
\address{Department of Mathematics, University of Illinois,
1409 West Green St., Urbana, IL 61801}
\email{tumanov@illinois.edu}

\thanks{Partially supported by Simons Foundation grant.}

\maketitle

\begin{abstract}
We consider a problem whether a CR mapping
of a generic manifold in complex space is uniquely
determined by its finite jet at a point, which is
referred to as finite jet determination.
We derive the finite jet determination for CR mappings
of smooth Levi nondegenerate manifolds of
arbitrary codimension from the finite
dimensionality of the algebra of infinitesimal
automorphisms of the corresponding quadrics.
Previously, this implication was known for
real analytic manifolds.
We prove a new 2-jet determination result
that covers most affirmative results on this matter
obtained so far.

\end{abstract}

\section{Introduction}

Let $M\subset\C^n$ be a generic real submanifold of real
codimension $k$. Then $M$ has CR dimension
$\dim_{CR}M=m=n-k$.
We introduce coordinates $(z,w)\in\C^n$,
$z\in\C^m$, $w=u+iv\in\C^k$, so that $M$ has a local
equation of the form
$$
v=h(z,u),
$$
here $h=(h_1,\dots,h_k)$ is a smooth real vector function with $h(0)=0$, $dh(0)=0$.
Furthermore, we can choose the coordinates so that the equation of $M$ takes the form
$$
v=h(z,u)=F(z,z)+O(|z|^3+|u|^3).
$$
Here
$$
F=(F_1,\dots, F_k),\quad
F_j(z,z)=\<A_j z,z\>,\quad
\<a,b\>=\sum a_l \bar b_l,
$$
$A_j$-s are Hermitian matrices.
The vector valued Hermitian from $F$
can be regarded as the Levi form of $M$ at $0$.

\begin{itemize}
\item
We say $M$ is (Levi) \emph{nondegenerate} at 0 if \br
(a) the matrices $A_j$ are linearly independent and \br
(b) $F(z,\zeta)=0$ for all $z\in\C^m$
implies $\zeta=0$.
\smallskip

\item
We say $M$ is \emph{strongly nondegenerate} at 0 if
$M$ is nondegenerate and
there is $c\in\R^k$ such that
$\det\left(\sum c_j A_j\right)\ne0$.
\smallskip

\item
We say $M$ is \emph{strongly pseudoconvex} at 0 if
$M$ is nondegenerate and
there is $c\in\R^k$ such that $\sum c_j A_j >0$.
\end{itemize}
We will also call the form $F$ itself respectively
nondegenerate, strongly nondegenerate,
or strongly pseudoconvex in these cases.

We recall that a \emph{CR mapping}
or a \emph{CR diffeomorphism}
between two generic submanifolds is a diffeomorphism
that induces a complex isomorphism between their
complex tangent bundles.

We are concerned with the problem whether a
CR diffeomorphism between two manifolds is uniquely
determined by its finite jet at a point,
which is referred to as
{\em finite jet determination}.
The problem has been popular since 1970-s,
and the number of publications on the matter
has been very large, see surveys
\cite{Baouendi, BC-M-survey, Zaitsev}.
Nevertheless, there are fundamental open questions even
in the Levi nondegenerate case, to which we restrict here.

Tanaka \cite{Tanaka} gave a solution to the CR equivalence
problem for nondegenerate CR manifolds of codimensions
$k=1, m^2-1, m^2$.
His result implies 2-jet determination
for real analytic CR mappings of
real analytic nondegenerate manifolds
of said codimensions.
Tanaka's result for a real hypersurface ($k=1$)
was rediscovered by Chern and Moser \cite{Chern-Moser}.

Beloshapka \cite{Beloshapka1988} proved
finite jet determination
for real analytic CR mappings of
real analytic nondegenerate manifolds.

Bertrand and Meylan \cite{BM2021} prove
2-jet determination for $C^3$-smooth CR mappings of
$C^4$-smooth generic nondegenerate
manifold $M$ with additional condition that the authors
call \emph{D-nondegeneracy}. In particular, it implies
that there is $z\in \C^m$ such that the vectors
$\{A_j z: 1\le j\le k\}$ are $\R$-linearly independent.
This condition is quite restrictive, in particular,
it implies that $k\le 2m$,
whereas the dimension of the space of all Hermitian
forms on $\C^m$ is equal to $m^2$, and the nondegeneracy
condition, part (a), imposes only the restriction
$k\le m^2$.

In \cite{T2020}, we prove
2-jet determination for smooth CR automorphisms of
smooth \emph{strongly pseudoconvex} manifolds.

Beloshapka \cite{Beloshapka2022} proves 2-jet determination
for nondegenerate codimension $\le 3$ manifolds in
the real analytic case.

On the other hand, Meylan \cite{M2020} has constructed a surprising counterexample
of a quadric for which 2-jet determination fails.
For arbitrary large integer $p$,
Gregerovi\v c and Meylan \cite{GM2020} have constructed
counterexamples for which $p$-jet determination fails.

In this paper, we prove 2-jet determination
for smooth CR mappings under a new more general condition
that we call \emph{T-nondegenaracy},
see Definition \ref{T-nondeg-definition}.
Our result implies all affirmative results on
2-jet determination mentioned above, in particular,
for strictly pseudoconvex, D-nondegenerate
and codimension $\le 3$ CR manifolds.

We note that
Bertrand, Blanc-Centi, and Meylan
\cite{BBM2019, BM2021, BM2022}
and the author \cite{T2020} obtained their results
by using stationary discs, a family of analytic discs
invariant under CR mappings. This method works so far only
for \emph{strongly} nondegenerate manifolds because
the existence of stationary discs is proven under
this assumption.

In this paper, we follow the classical approach
based on the infinitesimal automorphisms of quadrics.
The Lie algebra $\g$ of infinitesimal automorphisms
of the quadric $v=F(z,z)$ has a natural grading
$\g=\g_{-2}+\g_{-1}+\g_0+\ldots$.
It is long known \cite{Beloshapka1988,T1988} that for
a nondegenerate quadric, $\g$ has finite dimension,
which implies (Beloshapka \cite{Beloshapka1988})
finite jet determination in the real analytic case.
We prove (Theorem \ref{Finite-jet-det}) this fact in the smooth case.
In particular, we prove (Theorem \ref{g3=0})
that for T-nondegenerate
CR manifolds we have $\g_3=0$, which implies
(Theorem \ref{2-jet-det})
2-jet determination in the smooth case.

The paper is structured as follows.
In Section 2, we introduce infinitesimal automorphisms
of quadrics and state basic facts about them.
In Section 3, we develop a rudimentary theory
of weighted jets of mappings.
In Section 4, we derive finite jet determination
results for smooth CR mappings from the finite
dimensionality of $\g$.
In Section 5, we prove the vanishing of $\g_3$
for T-nondegenerate quadrics.
In Sections 6 and 7, we show that the previous
affirmative results on 2-jet determination
follow from our Theorems \ref{2-jet-det} and \ref{g3=0}.
Finally, in Section 8, we give several examples
illustrating various nondegeneracy conditions
and their effect on 2-jet determination.

The author wishes to thank Francine Meylan for letting him know about her counterexample \cite{M2020}
and Alexander Sukhov for useful discussions.

\section{Infinitesimal automorphisms of quadrics}
An infinitesimal CR-automorphism of a CR-manifold $M$
is a vector field on $M$ that generates a local 1-parameter
group of CR-mappings (CR-automorphisms) $M\to M$.

Let $M_0$ be a nondegenerate quadric defined
as before by the equations
$$
v=F(z,z),\quad
z\in\C^m,\quad
w=u+iv\in\C^k.
$$
Here
$$
F=(F_1,\dots, F_k),\quad
F_j(z,z)=\<A_j z, z\>,\quad
\<a,b\>=\sum a_l\bar b_l.
$$
Let $G$ be the group of all CR-automorphisms $M_0\to M_0$.
Then $G$ is a finite dimensional Lie group and its Lie algebra $\g$ is the set of all infinitesimal automorphisms of $M_0$.
The dimension of $G$ has an estimate depending on $m$ and $k$
(see Beloshapka \cite{Beloshapka1988},
Isaev and Kaup \cite{Isaev-Kaup},
and the author \cite{T1988}).

It turns out that all
elements of $G$ and $\g$ are respectively rational and polynomial.
In particular, every vector field $X\in\g$ has the form
\begin{equation*}
X=\sum f_j\frac{\d}{\d z_j}+\sum g_\l\frac{\d}{\d w_\l}
=f\frac{\d}{\d z}+g\frac{\d}{\d w}
=:(f,g),
\end{equation*}
where $f$ and $g$ are polynomial vector functions
of $z$ and $w$ that satisfy the equation
(see \cite{Beloshapka1988, BC-M-survey})
\begin{equation}\label{inf-auto-equ}
\Im(g-2iF(f,z))=0,\quad
(z,w)\in M.
\end{equation}
This equation implies
(see \cite{Beloshapka1988, BC-M-survey})
\begin{equation}\label{z-degree}
\deg_z f\le 2, \quad
\deg_z g\le 1.
\end{equation}
We will use the equation (\ref{inf-auto-equ})
rather than the infinitesimal automorphisms themselves.

We give the variables and differentiations
$z_j, w_j, \d/\d z_j, \d/\d w_j$ the weights
$1,2,-1,-2$ respectively.
Let $\g_p$ be the set of vector fields $X\in\g$ with
weighted homogeneous degree $p\in\Z$. Then
the elements $X=(f_{p+1},g_{p+2})\in\g_p$ are
solutions of the equation
\begin{equation}\label{g_p}
\Im(g_{p+2}-2iF(f_{p+1},z))|_{w=u+iF(z,z)}=0,
\end{equation}
which is a linear homogeneous equation on the
coefficients on the weighted homogeneous
polynomials $f_{p+1}$ and, $g_{p+2}$ of weighted
degrees respectively $p+1$ and $p+2$.
Then
$$
\g=\sum_{p=-2}^\infty \g_p
$$
is a graded Lie algebra, that is,
$[\g_p,\g_q]\subset\g_{p+q}$.
Since $\g$ has finite dimension, $\g_p=0$ for big $p$.
Since $F$ is nondegenerate, it follows that
each vector $\xi\in\g_p$ is uniquely determined
by the map $\ad\xi:\g_{-1}\to \g_{p-1}$,
here $(\ad\xi)(\eta)=[\xi,\eta]$.
In particular, if $\g_p=0$, then $\g_q=0$ for all $q>p$.

\section{Weighted jets of mappings}
A CR manifold $M$ has a complex tangent subbundle
$T^c(M)\subset T(M)$.
In this section, we consider a more general structure.
Let $M$ be a smooth real manifold
of dimension $n$
with a real vector subbundle $H(M)\subset T(M)$ with
fiber dimension $m$.
(We will specify the smoothness class later.)
For brevity we will call such manifolds \emph{H-manifolds}.
We will consider \emph{morphisms} of H-manifolds,
that is, mappings $M\to M'$ whose tangent mappings
take $H(M)$ to $H(M')$.
We introduce a rudimentary theory
of weighted jet bundles of morphisms of H-manifolds.

Let $M$ be a H-manifold. Let $a\in M$.
Let $(x,y)\in\R^m\times\R^k$, $k=n-m$,
be local coordinates such that $a$
has coordinates $(0,0)$,
and $H_a(M)$ is represented by the equation $y=0$.
We will call such a coordinate system \emph{admissible}
with origin at $a\in M$.
One should think of $x$ and $y$ variables as having weights
respectively 1 and 2. We will use weighted distance
to the origin $\sigma=|x|+|y|^{1/2}$.

Let $M$ and $M'$ be H-manifolds, and let $a\in M$,
$a'\in M'$. Let $(x,y)$ and $(x',y')$ be corresponding admissible coordinates.
Let $\Phi:M\to M'$ be a morphism, $\Phi(a)=a'$.
Then $\Phi$ can be represented by equations
$$
x'=\phi(x,y), \quad
y'=\psi(x,y).
$$
Then we will write $\Phi\cong(\phi,\psi)$.
Since $\Phi(a)=a'$, we have $\phi(0,0)=0$, $\psi(0,0)=0$.

Let $\Phi\cong(\phi,\psi)$ and
$\tilde\Phi\cong(\tilde\phi,\tilde\psi)$
be morphisms sending $a$ to $a'$.
Let $p,q$ be nonnegative integers with $p\le q\le 2p$.
We say that $\Phi$ and $\tilde\Phi$ are $(p,q)$-equivalent
and write $\Phi\sim_{(p,q)}\tilde\Phi$ if
$$
\phi-\tilde\phi=o(\sigma^p), \quad
\psi-\tilde\psi=o(\sigma^q).
$$
We claim that the equivalence relation $\sim_{(p,q)}$
is well defined, that is, independent of coordinates
$(x,y)$ and $(x',y')$. Without loss of generality
we restrict to the case $\tilde\Phi\equiv a'$, that is,
$\tilde\phi\equiv0, \tilde\psi\equiv0$.

We first show the independence of  $(x,y)$.
Indeed, let $(\hat x,\hat y)$ be another
admissible system. Then $\d\hat y/\d x=0$,
$\hat x=O(|x|+|y|)$, $\hat y=O(|x|^2+|y|)$.
Then
$\hat\sigma=|\hat x|+|\hat y|^{1/2}
=O(|x|+|y|)+O(|x|+|y|^{1/2})=O(\sigma)$.
Similarly, $\sigma=O(\hat\sigma)$.
Hence, the relation is independent of $(x,y)$.

We now consider the effect of another coordinate
system $(\hat x',\hat y')$ in the target.
Let $\phi=o(\sigma^p)$, $\psi=o(\sigma^q)$.
Then in the new coordinates we will have
$\hat\phi=O(|\phi|+|\psi|)
=o(\sigma^p+\sigma^q)=o(\sigma^p)$ because $p\le q$.
We will also have
$\hat\psi=O(|\phi|^2+|\psi|)
=o(\sigma^{2p}+\sigma^q)=o(\sigma^q)$
because $q\le 2p$.
Hence, the relation is independent of $(x',y')$.

We denote the set of equivalence classes
of morphisms sending $a$ to $a'$
by $J^{p,q}_{a,a'}(M,M')$ and the set of
all classes for all pairs $(a,a')$ by
$J^{p,q}(M,M')$. Then one can see that
$J^{p,q}(M,M')\to M\times M'$
is a fiber bundle whose fiber is a Euclidean space.
We call the members of $J^{p,q}(M,M')$ weighted
jets of morphisms $M\to M'$.

Let $\Phi:M\to M'$ be a morphism with
$\Phi(a)=a'$. We denote by $j^{p,q}_a(\Phi)$
the equivalence class of $\Phi$ in
$J^{p,q}_{a,a'}(M,M')$ and call it the
$(p,q)$-jet of $\Phi$ at $a$.
Given coordinates $(x,y)$ and $(x',y')$,
for $\Phi\cong(\phi,\psi)$,
this jet is represented by the set
$$
\{
a,a', D^\alpha\phi(0), D^\beta\psi(0):
0<\|\alpha\|\le p, 0<\|\beta\|\le q
\},
$$
here $\alpha$ and $\beta$ are multi-indices
corresponding to the variables $(x,y)$, and
for a multi-index $\alpha=(\alpha_x,\alpha_y)$,
the weighted length
$\|\alpha\|=|\alpha_x|+2|\alpha_y|$.

For a morphism $\Phi:M\to M'$, we also introduce
the map $j^{p,q}(\Phi):M\to J^{p,q}(M,M')$,
$a\mapsto j^{p,q}_a(\Phi)$, and call it
the $(p,q)$-jet of $\Phi$.

We now restrict to the case $q=p+1$ and in all notations
introduced so far, we replace the pair $(p,q)$ by
the single number $p$.
In our applications to CR manifolds, we will
understand weighted jets of mappings this way.
The main point of this section is the following.

\begin{prop}\label{Derivative-of-jet}
Let $p\ge 1$ be an integer.
Let $M$ and $M'$ be $C^{p+2}$-smooth H-manifolds.
Let $\Phi:M\to M'$ be a morphism.
Let $a\in M$, $a'=\Phi(a)$, and let $X\in H_a(M)$.
Then $X j^p(\Phi)$ (the derivative of the mapping
along the tangent vector) can be expressed
in terms of $j^{p+1}_a(\Phi)$.
\end{prop}

The result is quite intuitive because the derivative
with respect to the $x$-variable, whose weight is 1,
is supposed to add 1 to the weighted order of the jet.
If the bundles $H(M)$ and $H(M')$ were integrable
(involutive) then we could represent all jets in the
same coordinate systems. Then the conclusion would be
obvious. However, in our application to CR manifolds,
they are not integrable, so the result needs
some attention.

\begin{proof}
We fix admissible coordinates $(x,y)$ and $(x',y')$
in $M$ and $M'$, and we call them main coordinates
identifying $M$ and $M'$ with open sets in the Euclidean
space. In order to represent $j^p(\Phi)$,
we introduce admissible coordinates
$(\xi,\eta)$ and $(\xi',\eta')$ with origins at
$(x,y)$ and $(x',y')$ respectively, and we call
them variable coordinates.
We define $(\xi,\eta)$ by the mapping
$$
(\xi, \eta)\mapsto (\tilde x, \tilde y)=
(x+\xi, y+\eta+A(x,y)\xi),
$$
here $A(x,y)$ is a matrix, so $H_{(x,y)}(M)$
in the main coordinates $(x,y)$ is defined as
$\{(\xi,A(x,y)\xi): \xi\in\R^m \}$.
Then $A(0,0)=0$.
We define  $(\xi',\eta')$ similarly.

Let $\Phi\cong(\phi,\psi)$ in the main coordinates.
We need the representation of $\Phi$ in the variable
coordinates, denote it $\Phi\cong(f,g)$. Then we have
\begin{align*}
&f=\phi(x+\xi,y+\eta+A(x,y)\xi)-\phi(x,y),\\
&g=\psi(x+\xi,y+\eta+A(x,y)\xi)-\psi(x,y)
-A'(\phi(x,y),\psi(x,y))f.
\end{align*}
The desired claim means exactly that the derivatives
\begin{align*}
&\d_x|_{(x,y)=(0,0)}
D^\alpha_{\xi,\eta}|_{(\xi,\eta)=(0,0)}f,\quad
\|\alpha\|\le p,\\
&\d_x|_{(x,y)=(0,0)}
D^\beta_{\xi,\eta}|_{(\xi,\eta)=(0,0)}g,\quad
\|\beta\|\le p+1,
\end{align*}
can be expressed in terms of
$D^\alpha\phi(0,0)$ and
$D^\beta\psi(0,0)$ with
$\|\alpha\|\le p+1$,
$\|\beta\|\le p+2$.
The verification is an exercise on the Chain Rule.
Indeed, an increase in the weighted order of
derivative by more than one unit can happen
when we differentiate $\phi$ or $\psi$
with respect to the variable $\xi$ that occurs in the
second argument of the function, because it adds 2
to the weighed order. Then the factor $A(x,y)$ comes out.
Then in the subsequent differentiation with
respect to $x$, we only have to differentiate
$A(x,y)$ because $A(0,0)=0$. Thus the final differentiation
with respect to $x$ does not increase
the weighted order of the derivatives of $\phi$ or $\psi$,
so the resulting weighted order
of the derivative increases by no more than one unit.

We point out that the choice $q=p+1$ is dictated by
the above requirement on the derivatives of $g$.
Since the equation of $g$ involves $f$,
the derivatives of $g$ would include the derivatives
of $\phi$, whose order has to be bounded by $p+1$.
We leave the details to the reader.
\end{proof}

\section{Finite jet determination}
Let $M$ be a nondegenerate
CR manifold with equation
$$
v=h(z,u)=F(z,z)+O(|z|^3+|u|^3),
$$
and let $M_0$ be the corresponding quadric with equation
$$
v=F(z,z).
$$
Let $\g$ be the graded Lie algebra of infinitesimal
automorphisms of $M_0$.
It turns out that the finite dimensionality of $\g$
implies finite jet
determination for smooth CR mappings of $M$.

A CR manifold $M$ is an H-manifold as defined in
Section 3 with $H(M)=T^c(M)$. Then we can use
weighted jets of CR mappings as we defined there.

\begin{thm}\label{Finite-jet-det}
Let $p\ge 1$ be an integer.
Let $M, M'$ be $C^{p+2}$-smooth nondegenerate CR manifolds
defined as above. Suppose $\g_{p}=0$.
Then every germ at 0 of a $C^{p+2}$-smooth CR diffeomorphism
$\Phi:M\to M'$
is uniquely determined by the weighted jet $j^p_0(\Phi)$ .
\end{thm}

Beloshapka \cite{Beloshapka1988} obtained the real analytic version of this result. As a consequence we obtain
the following.

\begin{thm}\label{2-jet-det}
Let $M, M'$ be $C^5$-smooth non-degenerate
CR manifolds defined as above. Suppose $\g_3=0$.
Then every germ at 0 of a $C^5$-smooth CR diffeomorphism
$\Phi:M\to M'$
is uniquely determined by the (standard)
2-jet of $\Phi$ at $0$.
\end{thm}

\begin{proof}
[Proof of Theorem \ref{Finite-jet-det}]
We first show that $j^{p+1}_0(\Phi)$ is uniquely
determined by $j^p_0(\Phi)$.
The argument is well known
(see \cite{Beloshapka1988, BC-M-survey}),
and we describe it briefly.

We recall that a $C^r$-smooth CR function $f$
on a $C^r$-smooth generic manifold in $\C^n$
defined as above can be developed into
a Taylor series with remainder of the form
$f(z,w)=P(z,w)+o(|z|^r+|w|^r)$, where $P(z,w)$
is a holomorphic polynomial of degree $\le r$.
We can represent
$P$ as a sum of weighted homogeneous
holomorphic polynomials, in which variables
$z$ and $w$ have weights respectively 1 and 2,
and the index means the weighted degree of
the polynomial.
Then we obtain
$$
f=f_0+f_1+\ldots+f_r+o(|z|^r+|w|^{r/2}).
$$

Following Chern and Moser \cite{Chern-Moser}
and Beloshapka \cite{Beloshapka1988},
we expand the equations of $M$ and $M'$ and
the CR mapping $\Phi=(f,g)$ into Taylor series with
remainders and represent them as sums of weighed homogeneous
components, where $z$, $w$, and $u$ have weights
respectively 1, 2, and 2.
Without loss of generality, $\Phi(0)=0$. Then we have
\begin{align*}
& v=h(z,u)=F+h_3+\ldots+h_{p+2}
+o(|z|^{p+2}+|u|^{(p+2)/2})\\
& v'=h'(z',u')=F'+h'_3+\ldots+h'_{p+2}
+o(|z'|^{p+2}+|u'|^{(p+2)/2}) \\
& z'=f(z,w)=f_1+f_2+\ldots+f_{p+2}
+o(|z|^{p+2}+|w|^{(p+2)/2}) \\
& w'=g(z,w)=g_2+g_3+\ldots+g_{p+2}
+o(|z|^{p+2}+|w|^{(p+2)/2}).
\end{align*}
Since the derivative of $\Phi$ maps the complex
tangent plane $w=0$ to the plane $w'=0$,
we have $g_1=0$.
By linear transformations of $z$ and $w$,
without loss of generality,
we can put $f_1=z$, $g_2=w+Q(z)$, where
$Q$ is a quadratic polynomial, but one can see that
$Q=0$. Also, one can see that $F'=F$.

By plugging $z'$ and $w'$ in terms of $z$ and $w=u+ih(z,u)$
in the equation of $M'$ and collecting the terms
of weight $p+2$, we obtain:
\begin{equation}\label{nonhomo-g_p}
\Im(g_{p+2}-2iF(f_{p+1},z))|_{w=u+iF(z,z)}=\ldots
+o(|z|^{p+2}+|u|^{(p+2)/2}),
\end{equation}
here the dots mean terms that include only $f_{q+1}$ and
$g_{q+2}$ with $q<p$.

Note that the corresponding homogeneous
equation (\ref{g_p}) has only the trivial solution
because $\g_p=0$.
Then the component $(f_{p+1},g_{p+2})$ is uniquely
determined by the components of $\Phi$ of lower
weighted degree. That is, $j^{p+1}_0(\Phi)$ is uniquely
determined by $j^p_0(\Phi)$.


We note that since $\g_p$ is obtained by solving a
linear homogeneous system (\ref{g_p}),
the hypothesis that $\g_p=0$ holds at zero implies
that $\g_p=0$ holds at nearby points even though
the Levi form of $M$ at those points is not
necessarily equivalent to $F$ as a vector valued
Hermitian form, but it will be close to $F$.
Without loss of generality,
$\g_p=0$ for every point of $M$.

Let $\gamma:(-1,1)\to M$ be a CR curve in $M$,
that is, $d\gamma(t)/dt\in T^c_{\gamma(t)}(M)$,
and let $\gamma(0)=0$.
By the above arguments, $j^{p+1}_{\gamma(t)}(\Phi)$
is a function of $j^p_{\gamma(t)}(\Phi)$.
By Proposition \ref{Derivative-of-jet}, this gives rise
to a first order ordinary differential equation
on the jet $\Psi(t)=j^p_{\gamma(t)}(\Phi)$.
This equation has a unique solution with given
initial condition $\Psi(0)=j^p_0(\Phi)$.
In particular, $\Phi(\gamma(t))$ is uniquely determined
for all $t$.

Since $M$ is nondegenerate, CR curves through 0
cover an open set in $M$. Then $\Phi$ is uniquely
determined by $j^p_0(\Phi)$, as desired.
\end{proof}

\begin{proof}[Proof of Theorem \ref{2-jet-det}]
By Theorem \ref{Finite-jet-det}, the mapping $\Phi$
is uniquely determined by the the weighted jet
$j^3_0(\Phi)$. We note that $j^3_0(\Phi)$
is uniquely determined by the standard 2-jet
of $\Phi$ at 0. This follows by (\ref{nonhomo-g_p})
with $p=2$ and the bounds (\ref{z-degree})
on the $z$-degrees of $f_3$ and $g_4$.
(See \cite{BC-M2022,BC-M-survey} for details.)
\end{proof}

\begin{remark}
{\rm
The regularity requirements in Theorems \ref{Finite-jet-det}
and \ref{2-jet-det} are set so that all occurring jets
make sense. We note that by Stein \cite{Stein},
a $C^r$-smooth CR function $f$ on a nondegenerate
generic CR manifold has roughly $2r$ derivatives
along complex tangential directions. Then a weighted
jet $j^p f$ makes sense for some $p>r$.
Thus the required smoothness of class $C^{p+2}$
in Theorem \ref{Finite-jet-det}
in fact can be lowered. We leave the details to the reader.
}
\end{remark}

\section{Vanishing of $\g_3$}
Let $F$ be as above.
Let $S\subset \C^m$ be a set.
We define the orthogonal complement
$$
S^F=\{x\in\C^m: \forall y\in S, F(x,y)=0\}.
$$
Then $S^F$ is a complex subspace of $\C^m$.
Similarly, we can define $S^{\Re F}$.
We note that $S^F\subset S^{\Re F}$,
and if $S$ is a complex subspace,
then $S^F=S^{\Re F}$.
For $a,b\in\C^m$, we define complex subspaces
\begin{equation}\label{T(a,b)}
T(a,b)=\{y\in\C^m: \exists x\in\C^m: F(x,b)+F(a,y)=0\},\quad
T(a)=T(a,a).
\end{equation}

We observe
\begin{equation}\label{g3-0}
T(a,b)\supset a^F+\C b, \quad
T(a,b)^F\subset (a^F)^F\cap b^F, \quad
T(a)^F\subset (a^F)^F\cap a^F
\end{equation}

If a certain property holds for every element $x$
in an open dense set of a space $X$, then we will say that
the property holds for a {\it generic} element of $X$.

\begin{definition}\label{T-nondeg-definition}
We say that a nondegenerate form $F$ is
{\em T-nondegenerate} if for generic $a\in \C^m$,
we have $T(a)^F=0$.
\end{definition}

\begin{thm}\label{g3=0}
Let $F$ be a T-nondegenerate form.
Then $\g_3=0$.
\end{thm}

In the proof, we will apply the following lemma
three times.

\begin{lemma}\label{g3-lemma}
Let $\phi:\C^k\to\C^m$ be a $C^1$ mapping.
Suppose for every $a,b\in\C^m$, we have $F(\phi(F(a,b)),b)=0$.
Then for every $a,b\in\C^m$, we have
$\phi(F(a,b))\in T(a,b)^F$.
In particular, for every $a\in\C^m$, we have
$\phi(F(a,a))\in T(a)^F$.
Hence, if $F$ is T-nondegenerate, then for every
$a\in\C^m$, we have $\phi(F(a,a))=0$.
\end{lemma}

\begin{proof}
Let $a, b, x, y\in\C^m$,
let $F(x,b)+F(a,y)=0$, so $y\in T(a,b)$.
For $t\in\R$, we put
$$
a'=a+tx,\quad
b'=b+ty.
$$
Then we have
\begin{align*}
&F(a',b')=F(a,b)+t(F(x,b)+F(a,y))+t^2 F(x,y)
=F(a,b)+O(t^2),\\
&\phi(F(a',b'))=\phi(F(a,b))+O(t^2).
\end{align*}
By the hypothesis we have
\begin{align*}
0 &=F(\phi(F(a',b')),b')\\
&=F(\phi(F(a,b))+O(t^2),b+ty)\\
&=F(\phi(F(a,b),b)+tF(\phi(F(a,b)),y)+O(t^2).
\end{align*}
Then $F(\phi(F(a,b)),y)=0$. Since it holds for every
$y\in T(a,b)$, we have $\phi(F(a,b))\in T(a,b)^F$,
as desired.
\end{proof}

\begin{proof}[Proof of Theorem \ref{g3=0}]
An element $(f,g)\in\g_3$ has the following form
(see \cite{Beloshapka1988, BC-M-survey})
\begin{align*}
& f(z,w)=A(z,z,w)+B(w,w),\\
& g(z,w)=2iF(z, B(\bar w, \bar w)).
\end{align*}
Here $A$ and $B$ are complex multilinear forms
such that
$A$ is symmetric in the first two arguments
and $B$ is symmetric. They are characterized by
the following equations:
\begin{align}
& F(A(z,z,F(z,b)),b)=0, \label{g3-1}\\
& F(A(z,z,w),b)=4iF(z, B(\bar w, F(b,z))).\label{g3-2}
\end{align}
Here $b,z\in\C^m$, and $w\in\C^k$ are arbitrary.
In fact, $A(z,z,F(z,b))=0$, but we do not
need this fact here.
We would like to show that if $F$ is T-nondegenerate,
then $A=0$ and $B=0$.

We first plug $w=F(z,b)$ in (\ref{g3-2})
and using (\ref{g3-1}) we obtain
\begin{equation}\label{g3-3}
F(z, B(F(b,z),F(b,z)))=0.
\end{equation}
Put $\phi(w)=B(w,w)$. Then by (\ref{g3-3})
we have $F(\phi(F(b,z)),z)=0$ for all $b,z\in\C^m$.
Since $F$ is T-nondegenerate, by Lemma \ref{g3-lemma},
we have $\phi(F(b,b))=0$, that is,
\begin{equation*}
B(F(b,b),F(b,b))=0.
\end{equation*}
By polarization
\begin{equation*}
B(F(b,z),F(b,z))=0.
\end{equation*}
Since $B$ is symmetric, we have
\begin{equation}\label{g3-4}
B(F(b,a),F(b,z))=0
\end{equation}
for all $a,b,z\in\C^m$.
Plug $w=F(a,b)$ in (\ref{g3-2}). Then by (\ref{g3-4})
we have
\begin{equation}\label{g3-5}
F(A(z,z,F(a,b)),b)=4iF(z, B(F(b,a), F(b,z)))=0.
\end{equation}
Put $\phi(w)=A(z,z,w)$, where $z\in\C^m$ is constant.
Then by (\ref{g3-5}) we have
$F(\phi(F(a,b)),b)=0$. By Lemma \ref{g3-lemma}, we have
$\phi(F(a,a))=0$, that is,
$$
A(z,z,F(a,a))=0.
$$
Since $F$ is nondegenerate, the values $F(a,a)$ span
all of $\C^k$, hence $A=0$. Then by (\ref{g3-2})
we have
\begin{equation}\label{g3-6}
F(z, B(u,F(b,z)))=0
\end{equation}
for all $b,z\in\C^m$ and $u\in\C^k$.
Put $\phi(w)=B(u,w)$, where $u\in\C^k$ is constant.
Then by (\ref{g3-6}), we have
$F(\phi(F(b,z)),z)=0$. By Lemma \ref{g3-lemma},
we have $\phi(F(b,b))=0$, that is,
$$
B(u,F(b,b))=0.
$$
Hence $B=0$. Now $A=0$, $B=0$, $f=0$, $g=0$,
and $\g_3=0$, as desired.
\end{proof}

\section{Other conditions}

Let $F$ and $A_j$ be as above.
Let $F$ be nondegenerate. Recall that $F$ is strongly
pseudoconvex if there is $c\in\R^k$ such that
$A=\sum_{j=1}^{k} c_j A_j>0$, positive definite.
In \cite{T2020} we prove 2-jet determination for
strongly pseudoconvex CR manifolds.
We recover this result using the following.
\begin{prop}
Let the form $F$ be strongly pseudoconvex.
Then $F$ is T-nondegenerate, hence $\g_3=0$.
\end{prop}
\begin{proof}
If $F$ is strongly pseudoconvex, then for $z\ne 0$ we have
$F(z,z)\ne 0$. Then we have $(a^F)^F\cap a^F=0$, because
if $z\in (a^F)^F\cap a^F$, then $F(z,z)=0$. Now the result reduces to the following result.
\end{proof}

\begin{prop}\label{Prop-orth-compl-imply-T}
Suppose for generic $a\in\C^m$, we have $(a^F)^F\cap a^F=0$. Then $F$ is T-nondegenerate.
\end{prop}
\begin{proof}
By (\ref{g3-0}), we have
$T(a)^F\subset (a^F)^F\cap a^F=0$.
Hence $F$ is T-nondegenerate, as desired.
\end{proof}

Bertrand, Blanc-Centi, and Meylan
\cite{BBM2019},
prove 2-jet determination for
so called {\em fully nondegenerate} CR manifolds
in the smooth case.
Their condition, in particular, implies that there is
$a\in\C^m$ such that
the vectors $(A_j a)_{j=1}^k$ are $\C$-linear independent.
Bertrand and Meylan \cite{BM2022} show that this weaker condition alone implies 2-jet determination in the
analytic case. We recover these results here by
proving the following.

\begin{prop}\label{Fully}
Suppose there is $a\in\C^m$ such that
the vectors $(A_j a)_{j=1}^k$ are $\C$-linear independent.
Then $F$ is T-nondegenerate, hence $\g_3=0$.
\end{prop}
\begin{proof}
By the hypotheses, for generic $a\in\C^m$, we have
$F(\C^m,a)=\C^k$. Then $T(a)=\C^m$, and $T(a)^F=0$,
hence $F$ is T-nondegenerate.
\end{proof}

Beloshapka \cite{Beloshapka2022} proves a sufficient
condition for $\g_3=0$ that in particular includes
the hypothesis that there is $a\in \C$ such that
the vectors $(A_j a)_{j=1}^k$
span $\C^m$ over $\C$. We observe that this hypothesis
alone suffices for $\g_3=0$.

\begin{prop}
Suppose that there is $a\in \C^m$ such that
the vectors $(A_j a)_{j=1}^k$
span $\C^m$ over $\C$.
Then $F$ is T-nondegenerate, hence $\g_3=0$.
\end{prop}

\begin{proof}
The hypothesis is equivalent to the fact that
for generic $a\in \C^m$, we have $a^F=0$.
Then this proposition is a special
case of Proposition \ref{Prop-orth-compl-imply-T}.
\end{proof}

Suppose there is $c\in\R^k$ such that
$A=\sum_{j=1}^{k} c_j A_j$ is nonsingular,
that is, $F$ is strongly nondegenerate.
Let $a\in\C^m$.
Following Bertrand and Meylan \cite{BM2021}, we introduce
a $m\times k$ matrix
$$
D=(A_1a,\ldots,A_ka)
$$
and a $k\times k$ matrix
\begin{equation}\label{matrix-B}
B=D^*A^{-1}D,\quad
B_{lj}=\<A_lA^{-1}A_ja,a\>.
\end{equation}

\begin{definition}[\cite{BM2021}]
$F$ is called {\em D-nondegenerate} if there exist
$c\in\R^k$ and $a\in\C^m$ such that the matrix
$\Re B:=\frac{1}{2}(B+\bar{B})$ is nonsingular.
\end{definition}

Bertrand and Meylan \cite{BM2021}
prove 2-jet determination for
D-nondegenerate CR manifolds.
We recover their result by proving the following.

\begin{prop}\label{D-nondeg}
Suppose $F$ is D-nondegenerate.
Then $F$ is T-nondegenerate, therefore, $\g_3=0$.
\end{prop}

\begin{lemma}\label{aReF-ReF}
$T(a)^F\subset(a^{\Re F})^{\Re F}\cap a^{\Re F}$.
\end{lemma}

To compare with (\ref{g3-0}), we observe that
$(a^{\Re F})^{\Re F}\subset (a^F)^F$.
Indeed, $a^F\subset a^{\Re F}$ implies
$(a^{\Re F})^{\Re F}\subset(a^F)^{\Re F}$.
Since $a^F$ is a complex subspace, we have
$(a^F)^{\Re F}=(a^F)^F$,
hence $(a^{\Re F})^{\Re F}\subset (a^F)^F$.

\begin{proof}
By putting $a=b$ and $x=y$ in (\ref{T(a,b)}),
we obtain $a^{\Re F}\subset T(a)$. Then
\begin{equation*}
T(a)^F\subset(a^{\Re F})^F\subset (a^{\Re F})^{\Re F}.
\end{equation*}
Since $a\in T(a)$, we have
$T(a)^F\subset a^F\subset a^{\Re F}$, and the lemma
follows.
\end{proof}

\begin{lemma}\label{ReB-implies}
Suppose $\Re B$ is nonsingular.
Then $(a^{\Re F})^{\Re F}\cap a^{\Re F}=0$.
\end{lemma}
\begin{proof}
Define $V=\Span_\R\{A_ja: 1\le j\le k\}$.
We claim that $a^{\Re F}=V^{\bot_\R}$, the orthogonal
complement of $V$ with respect to the standard
\emph{real} inner product $\Re\<.,.\>$.
Indeed, $z\in a^{\Re F}$ means exactly that
$\Re\<A_ja,z\>=0$ for every $j$, that is, $z\in V^{\bot_\R}$.

We now claim $z\in (a^{\Re F})^{\Re F}$ iff $A_j z\in V$
for every $j$. Indeed, it means that $\Re\<A_j z,x\>=0$
for every $1\le j\le k$ and $x\in a^{\Re F}=V^{\bot_\R}$.
That is, $A_jz\in (V^{\bot_\R})^{\bot_\R}=V$ for every $j$.

Let $z\in (a^{\Re F})^{\Re F}\cap a^{\Re F}$.
Since $z\in (a^{\Re F})^{\Re F}$, we have
$A_j z\in V$ for all $j$. Recall $A=\sum c_jA_j$, $c\in\R^k$.
Then $Az\in V$. Then there is $\lambda\in\R^k$
such that $Az=\sum \lambda_j A_ja$.
Then $z=\sum \lambda_j A^{-1}A_ja$.

Since $z\in a^{\Re F}$, we have
$\Re\<A_l z,a\>=0$ for every $l$.
Using the above expression for $z$, we obtain
$\Re\<\sum_j\lambda_jA_l A^{-1}A_ja,a\>=0$.
By (\ref{matrix-B})
we rewrite $\Re\sum_j B_{lj}\lambda_j=(\Re B) \lambda=0$.

Since $\Re B$ is nonsingular, $\lambda=0$. Then $z=0$,
hence the conclusion.
\end{proof}

\begin{proof}[Proof of Proposition \ref{D-nondeg}]
We observe that the condition $\det\Re B\ne0$ is stable
under small perturbations of $a\in\C^m$, therefore,
by Lemmas \ref{aReF-ReF} and \ref{ReB-implies}, we have
$T(a)^F=0$ for generic $a\in\C^m$,
that is, $F$ is T-nondegenerate, as desired.
\end{proof}

We finally observe that in Tanaka's \cite{Tanaka} cases
$k=m^2-1$ and $k= m^2$, nondegenerate CR manifolds are T-nondegenerate. Indeed, for $k=m^2-1$, one can see
that for generic $a\in\C^m$ we have $a^F=0$, so by
Proposition \ref{Prop-orth-compl-imply-T} the manifold
is T-nondegenerate. Then the case $k= m^2$ also falls into this category, but it is also strongly pseudoconvex.

\section{Codimension $k\le 3$}

By Tanaka \cite{Tanaka} and Chern and Moser \cite{Chern-Moser}, 2-jet determination
holds for codimension $k=1$.
Blanc-Centi and Meylan \cite{BC-M2022} prove 2-jet determination for holomorphic CR mappings for $k=2$.
Beloshapka \cite{Beloshapka2022} proves that $\g_3=0$ for $k=3$.
We recover these results here by proving the following.

\begin{prop}\label{k=<3}
Let $F$ be a nondegenerate quadric with $k\le 3$.
Then $F$ is T-nondegenerate, hence $\g_3=0$.
\end{prop}

For $k\le 2$ we have a stronger result.

\begin{prop}\label{Prop-k=2}
Let $F$ be a nondegenerate quadric with $k\le 2$.
Then for generic $a\in\C^m$, we have
$(a^F)^F\cap a^F=0$.
\end{prop}

For $k=1$ the result is obvious.
Indeed, for $k=1$ we have $(a^F)^F=\C a$.
If $F(a,a)\ne 0$, then $(a^F)^F\cap a^F=0$.

\begin{lemma}\label{LemmaL}
Let $L$ be a $m\times m$ matrix such that for
all $a,b\in \C^m$,
\begin{equation}\label{Lab}
  \<La,b\>^2=\<a,b\>\<L^2a,b\>.
\end{equation}
Then $L=\lambda I$, a scalar matrix.
\end{lemma}
\begin{proof}
The equation (\ref{Lab}) is invariant under similarity
$L\to PLP^{-1}$, $a\to Pa$, $b\to P^{*{-1}}b$.
Then without loss of generality we can assume
that $L$ is in Jordan normal form.

We first show that there are no Jordan cells
of order greater than 1, that is, $L$ is diagonal.
Suppose, otherwise, that there is a Jordan cell
of order greater than 1. Without loss of generality
we assume that $L$ consists of a single Jordan cell
$L=\lambda I+J$, here $J$ is the standard Jordan cell
with zero eigenvalue and units above the main diagonal.

Then after some cancellations, (\ref{Lab}) reduces to
$\<Ja,b\>^2=\<a,b\>\<J^2a,b\>$.
For $b=e_1$, the first element
of the standard basis, this equation turns into
$a_2^2=a_1 a_3$ ($a_2=0$ if $m=2$),
which is absurd because $a$ is arbitrary.
Hence $L$ is diagonal.

Let $L$ be diagonal with eigenvalues
$\lambda_1,\ldots,\lambda_m$.
Then we have
$$
\<a,b\>=\sum a_j \bar{b_j},\quad
\<La,b\>=\sum \lambda_j a_j \bar{b_j},\quad
\<L^2a,b\>=\sum \lambda_j^2 a_j \bar{b_j}.
$$
Put $c_j=a_j \bar{b_j}$. Then (\ref{Lab})
turns into
$$
\(\sum \lambda_j c_j\)^2
=\(\sum c_j\) \(\sum \lambda_j^2 c_j\).
$$
Put $c_1=c_2=1$, $c_j=0$ for $j>2$.
Then $(\lambda_1+\lambda_2)^2=2(\lambda_1^2+\lambda_2^2)$,
which implies $\lambda_1=\lambda_2$.
Similarly, all $\lambda_j$-s are equal,
hence $L$ is scalar, as desired.
\end{proof}

The following lemma must be well known.
\begin{lemma}\label{lemma-Aa-Ba-lin-indep}
Let $A$ and $B$ be linearly independent $m\times m$
Hermitian matrices.
Then for generic $a\in\C^m$, the vectors
$Aa$ and $Ba$ are $\C$-linearly independent.
\end{lemma}

For completeness, we include a proof.

\begin{proof}
Arguing by contradiction, we assume that for all
$z\in\C^m$, the vectors $Az$ and $Bz$ are linearly
dependent.

Let $r=\max \{\rank(\alpha A+\beta B): \alpha,\beta\in\R\}$.

We first consider the case $r=m$. For definiteness,
$A$ is nonsingular. Then for all $z\in\C^m$, there is
$\lambda\in\C$ such that $Bz=\lambda Az$, that is,
$A^{-1}Bz=\lambda z$. Then every $z\in\C^m$ is
an eigenvector of the matrix $A^{-1}B$. It can occur
only if $A^{-1}B$ is a scalar matrix.
Then $A^{-1}B=\lambda I$, and $B=\lambda A$,
which is absurd. Thus for $r=m$ the lemma holds.

We now consider the case $r<m$.
Without loss of generality assume $A$ is diagonal and
$\rank A=r$. We represent $A$ and $B$ as $2\times 2$
block matrices $A=(A_{pq})$, $B=(B_{pq})$, $p,q=1,2$
according to splitting the coordinates into two groups
of $r$ and $m-r$ coordinates. We number the coordinates
so that the $r\times r$ matrix $A_{11}$ is nonsingular
and the rest $A_{pq}=0$.
Since $\rank (tA+B)\le r$ for big $t\in\R$,
we must have $B_{22}=0$.

By the hypothesis, for every $z=(z_1,z_2)\in\C^m$,
we have $Bz=\lambda Az$ or $Az=0$.
The equation $Az=0$ means that $z_1=0$.
The equation $Bz=\lambda Az$, in particular, implies $B_{21}z_1=0$.
Then for every $z_1\ne 0$ we have $B_{21}z_1=0$.
Then $B_{21}=0$. Since $B$ is Hermitian, we have
also $B_{12}=0$. Then the result reduces to the case
$r=m$ for the matrices $A_{11}$ and $B_{11}$
that we considered earlier.
\end{proof}

\begin{proof}[Proof of Proposition \ref{Prop-k=2}]
Let $k=2$.
We change notation from $A_1$ and $A_2$ to $A$ and $B$.

As in the proof of Lemma \ref{lemma-Aa-Ba-lin-indep},
let $r=\max \{\rank(\alpha A+\beta B): \alpha,\beta\in\R\}$.

We first consider the strongly nondegenerate case $r=m$.
For definiteness, $A$ is nonsingular.
Arguing by contradiction, we assume that for generic
$a\in\C^m$ we have $(a^F)^F\cap a^F\ne 0$.

Put $U=\Span_\C(Aa,Ba)$.
We claim that $a^F=U^\bot$, here $U^\bot$ stands for the
orthogonal complement of $U$ with respect to the standard
{\em complex} inner product. We claim that $z\in (a^F)^F$ iff
$Az,Bz\in U$. The proofs of these claims are similar to
the proofs of the corresponding claims in the proof of
Lemma \ref{ReB-implies}.

Let $z\in (a^F)^F\cap a^F$, $z\ne 0$.
Since $z\in (a^F)^F$, we have
$Az=\alpha Aa+\beta Ba$ with $\alpha,\beta\in\C$.
Since $z\in a^F$, we have
$0=\<Az,a\>=\alpha \<Aa,a\>+\beta \<Ba,a\>$.
For generic $a$ we can assume $\<Aa,a\>\ne 0$.
Then $\alpha=-\beta\<Aa,a\>^{-1}\<Ba,a\>$.
Then $Az=-\beta(\<Aa,a\>^{-1}\<Ba,a\>Aa-Ba)$.
Since $A$ is nonsingular, we have $Az\ne0$, and $\beta\ne0$.
Without loss of generality $\beta=-\<Aa,a\>$.
Then
$Az=\<Ba,a\>Aa-\<Aa,a\>Ba$.
Then $z=\<Ba,a\>a-\<Aa,a\>A^{-1}Ba$.
Then $Bz=\<Ba,a\>Ba-\<Aa,a\>BA^{-1}Ba$.

Since $z\in a^F$, we have
$0=\<Bz,a\>=\<Ba,a\>^2-\<Aa,a\>\<BA^{-1}Ba,a\>$.

Since the equation $\<Ba,a\>^2=\<Aa,a\>\<BA^{-1}Ba,a\>$
holds for generic $a$, it holds for all $a\in\C^m$.
By polarization, we have
$\<Ba,b\>^2=\<Aa,b\>\<BA^{-1}Ba,b\>$.
By replacing $a$ by $A^{-1}a$, we have
$\<BA^{-1}a,b\>^2=\<a,b\>\<BA^{-1}BA^{-1}a,b\>$.

Put $L=BA^{-1}$.
Then we have $\<La,b\>^2=\<a,b\>\<L^2a,b\>$
for all $a,b\in \C^m$.
By Lemma \ref{LemmaL}
$L=\lambda I$ then $B=\lambda A$, which is absurd.
Thus the result holds in the strongly nondegenerate case
$r=m$.

Now consider the case $r<m$, that is,
every linear combination of $A$ and $B$ is singular.
Without loss of generality assume $A$ is diagonal and
$\rank A=r$.

As in the proof of Lemma \ref{lemma-Aa-Ba-lin-indep},
we represent $A$ and $B$ as $2\times 2$
block matrices $A=(A_{pq})$, $B=(B_{pq})$, $p,q=1,2$.
We assume that the $r\times r$ matrix $A_{11}$ is nonsingular
and the rest $A_{pq}=0$.
Then we again have $B_{22}=0$.
Since $F$ is nondegenerate, $B_{21}\ne0$.

We take $a=(a_1,a_2)\in\C^m$ such that
$\<Aa,a\>=\<A_{11}a_1,a_1\>\ne 0$ and $B_{21}a_1\ne 0$.
We prove
$(a^F)^F\cap a^F=0$.

Let $z=(z_1,z_2)\in (a^F)^F\cap a^F$.
As above, we have
$Az=\alpha Aa+\beta Ba$ with $\alpha,\beta\in\C$ and
$0=\<Az,a\>=\alpha \<Aa,a\>+\beta \<Ba,a\>$. In more detail,
\begin{equation*}
Az=\begin{pmatrix}
  A_{11}z_1 \\
  0
\end{pmatrix}
=\alpha\begin{pmatrix}
         A_{11}a_1 \\
         0
       \end{pmatrix}
+\beta\begin{pmatrix}
        \ldots \\
        B_{21}a_1
      \end{pmatrix}.
\end{equation*}
Since $B_{21}a_1\ne0$, we have $\beta=0$.
Then $\alpha \<Aa,a\>=0$.
Since $\<Aa,a\>\ne 0$, we have $\alpha=0$.
Then $Az=0$, that is, $A_{11}z_1=0$,
hence $z_1=0$.

Similarly, we have
$Bz=\gamma Aa+\delta Ba$ with $\gamma,\delta\in\C$ and
$0=\<Bz,a\>=\gamma \<Aa,a\>+\delta \<Ba,a\>$.
Since $z_1=0$, we have
\begin{equation*}
Bz=\begin{pmatrix}
  B_{12}z_2 \\
  0
\end{pmatrix}
=\gamma\begin{pmatrix}
         A_{11}a_1 \\
         0
       \end{pmatrix}
+\delta\begin{pmatrix}
        \ldots \\
        B_{21}a_1
      \end{pmatrix}.
\end{equation*}
By the same argument as above, we have
$\delta=\gamma=0$, and $Bz=0$.
Since $F$ is nondegenerate, $Az=Bz=0$ implies
$z=0$, and $(a^F)^F\cap a^F=0$, as desired.
\end{proof}

\begin{proof}[Proof of Proposition \ref{k=<3}]
By Proposition \ref{Prop-k=2}, it only remains to consider
the case $k=3$.
Let $r=\max\{\rank_\C(A_1a, A_2a, A_3a): a\in\C^m\}$.
By Lemma \ref{lemma-Aa-Ba-lin-indep}, the case
$r=1$ can not occur.

If $r=3$, then for generic $a$, the vectors
$A_1a$, $A_2a$, and $A_3a$ are $\C$-linearly independent.
Then by Proposition \ref{Fully},
$F$ is T-nondegenerate.

Finally, consider the case $r=2$.
For definiteness, assume that for generic $a\in\C^m$,
the vectors $A_1a$ and $A_2a$ are $\C$-linearly independent,
then $A_3a$ is a $\C$-linear combination of $A_1a$ and $A_2a$.

We denote by $F'$ the form defined by the matrices
$A_1$ and $A_2$.
Then $F'$ is nondegenerate.
Indeed, let $a\in\C^m$ be a vector such that
for all $z\in\C^m$, we have $F'(z,a)=0$. For generic $z$,
the vector $A_3z$ is a $\C$-linear combination of
$A_1z$ and $A_2z$.
Then for generic $z$, we have $F(z,a)=0$. Then $F(z,a)=0$
for all $z$. Since $F$ is nondegenerate, $a=0$.
Thus $F'$ is nondegenerate.

By Proposition \ref{Prop-k=2}, for generic $a$, we have
$(a^{F'})^{F'}\cap a^{F'}=0$.
Since $A_3a$ is a $\C$-linear combination of
$A_1a$ and $A_2a$ for generic $a$, we have
$a^F=a^{F'}$ for generic $a$. Then for generic $a$,
we have
$(a^F)^F=(a^{F'})^F\subset (a^{F'})^{F'}$, then
$(a^F)^F\cap a^F \subset (a^{F'})^{F'}\cap a^{F'}=0$.
Now by Proposition \ref{Prop-orth-compl-imply-T},
$F$ is T-nondegenerate, as desired.
\end{proof}

\section{Examples}

\begin{example}
\label{Example1}
{\rm
Meylan \cite{M2020}
has found an example of a strongly nondegenerate quadric
for which $\g_4\ne 0$, so 2-jet determination fails.
Here $m=4, k=5$, $F=(F_1,\ldots,F_5)$.
\begin{align*}
&F_1(z,z)=\Re(z_1 \bar z_3+z_2 \bar z_4),\\
&F_2(z,z)=|z_1|^2,\\
&F_3(z,z)=|z_2|^2,\\
&F_4(z,z)=\Re(z_1 \bar z_2),\\
&F_5(z,z)=\Im(z_1 \bar z_2).
\end{align*}
In this example, $F$ is {\em not} T-nondegenerate, and
$a^F\cap (a^F)^F\ne0$ for generic $a$.
}
\end{example}

\begin{example}
{\rm
Let $m=4, k=4$, and $F=(F_1,F_2,F_3,F_4)$
from Example \ref{Example1}.
Then $F$ is {\em not} T-nondegenerate, but
one can see that $\g_3=0$, so 2-jet determination
takes place.
This example shows that T-nondegeneracy is not necessary
for $\g_3=0$ to hold.
}
\end{example}

\begin{example}
{\rm
Let $m=4, k=3$, and $F=(F_1,F_2,F_3)$
from Example \ref{Example1}.
Since $k=3$, the form $F$ is T-nondegenerate,
$\g_3=0$, so 2-jet determination takes place.
However, one can see that $F$ is {\em not}
D-nondegenerate and
$(a^F)^F \cap a^F \ne0$ for generic $a$.
This example shows that
Proposition \ref{Prop-k=2} does not necessarily
hold for $k=3$.
}
\end{example}

\begin{example}
{\rm
Monomial quadric.
We call $F$ monomial if all components of $F$
have the form
$\Re(z_p \bar{z_q})$
(in particular, $|z_p|^2$)
or $\Im(z_p \bar{z_q})$ ($p\ne q$),
and there are no repeated components.
A monomial quadric $F$ is nondegenerate iff each variable
$z_p$ occurs in at least one component of $F$.
For such a quadric, the codimension $k$ can be any integer
$\frac{m}{2}\le k\le m^2$.
One can see that if $a\in\C^m$ has no zero components,
then $(a^F)^F\cap a^F=0$. Hence $F$ is T-nondegenerate.
However, $F$ can not be D-nondegenerate if $k>2m$.
}
\end{example}

\begin{example}
{\rm
Monomial antisymmetric quadric of an odd dimension.
This is a special case of the previous example
in which $m\ge 3$ is an odd integer, and all
components of $F$ have the from
$\Im(z_p\bar{z_q})$.
Then $k\le \frac{m(m-1)}{2}$.
Note that $F$ is not {\em strongly} nondegenerate,
hence not D-nondegenerate.
Indeed, the matrices of the components of $F$
are antisymmetric, and every linear combination of them
is an antisymmetric matrix of an odd order, hence singular.
As a special case of the previous example, if $F$ is
nondegenerate, then $F$ is T-nondegenerate, and 2-jet determination holds. However, this result would be difficult
to obtain by means of stationary discs because $F$ is not strongly nondegenerate.
}
\end{example}

\bigskip
\bigskip
{\bf Conflict of interest statement:}
There is no conflict of interest.

\bigskip
{\bf Data availability statement:}
This manuscript has no associated data.


\begin{thebibliography}{99}

\bibitem{Baouendi}
M. S. Baouendi,
Strong unique continuation and finite jet determination for Cauchy-Riemann mappings.
Phase space analysis of partial differential equations, 17--28, Progr. Nonlinear Differential Equations Appl., 69, Birkhäuser Boston, Boston, MA, 2006.

\bibitem{Beloshapka1988}
V. K. Beloshapka,
Finite-dimensionality of the group of automorphisms of a real analytic surface. (Russian)
Izv. Akad. Nauk SSSR Ser. Mat. 52 (1988), 437--442; translation in Math. USSR-Izv. 32 (1989), 443--448.

\bibitem{Beloshapka2022}
V. K. Beloshapka,
On special quadrics.
Russ. J. Math. Phys. 29 (2022), 11--27.

\bibitem{BBM2019}
F. Bertrand, L. Blanc-Centi, and F. Meylan,
Stationary discs and finite jet determination for non-degenerate generic real submanifolds.
Adv. Math. 343 (2019), 910--934.
Addendum: Adv. Math. 383 (2021), 107701, 3 pp.

\bibitem{BM2021}
F. Bertrand and F. Meylan,
Nondefective stationary discs and 2-jet determination in higher codimension.
J. Geom. Anal. 31 (2021), 6292--6306.

\bibitem{BM2022}
F. Bertrand and F. Meylan,
The 1-jet determination of stationary discs attached to generic CR submanifolds.
Preprint, arXiv:2112.12176

\bibitem{BC-M2022}
L. Blanc-Centi and F. Meylan,
Chern-Moser operators and weighted jet determination problems in higher codimension.
Internat. J. Math. 33 (2022), Paper No. 2250045, 15 pp.

\bibitem{BC-M-survey}
L. Blanc-Centi and F. Meylan,
On nondegeneracy conditions for the Levi map in higher codimension: a survey.
Preprint, arXiv:1711.11481

\bibitem{GM2020}
J. Gregorovi\v c and F. Meylan,
Construction of counterexamples to the 2-jet determination Chern-Moser Theorem in higher codimension.
Preprint, arXiv:2010.10220

\bibitem{Chern-Moser}
S. S. Chern and J. K. Moser,
Real hypersurfaces in complex manifolds.
Acta Math. 133 (1974), 219--271.

\bibitem{Isaev-Kaup}
A. Isaev and W. Kaup,
Regularization of local CR-automorphisms of real-analytic CR-manifolds.
J. Geom. Anal. 22 (2012), 244--260.

\bibitem{M2020}
F. Meylan,
A counterexample to the 2-jet determination Chern-Moser theorem in higher codimension.
Preprint, arXiv: 2003.11783.

\bibitem{Stein}
E. M. Stein,
Singular integrals and estimates for the Cauchy-Riemann equations.
Bull. Amer. Math. Soc. 79 (1973), 440--445.

\bibitem{Tanaka}
N. Tanaka,
On generalized graded Lie algebras and geometric structures I. J. Math. Soc. Japan, 19 (1967), 215--254.

\bibitem{T1988}
A. Tumanov,
Finite dimensionality of the group of CR automorphisms of a standard CR manifold, and proper holomorphic mappings of Siegel domains. (Russian)
Izv. Akad. Nauk SSSR Ser. Mat. 52 (1988), 651--659; translation in Math. USSR-Izv. 32 (1989), 655--662.


\bibitem{T2020}
A. Tumanov,
Stationary discs and finite jet determination for CR mappings in higher codimension.
Adv. Math. 371 (2020), 107254, 11 pp.

\bibitem{Zaitsev}
D. Zaitsev,
Unique determination of local CR-maps by their jets: a survey. Harmonic analysis on complex homogeneous domains and Lie groups (Rome, 2001). Atti Accad. Naz. Lincei Cl. Sci. Fis. Mat. Natur. Rend. Lincei (9) Mat. Appl. 13 (2002), 295--305.

\end{thebibliography}
\end{document}